\documentclass{amsart}
\usepackage{amsthm,amsfonts,amssymb,amsmath}

\DeclareMathOperator{\id}{id}
\DeclareMathOperator{\BV}{BV}
\DeclareMathOperator{\Hom}{Hom}
\DeclareMathOperator{\IBL}{IBL}
\DeclareMathOperator{\Sh}{Sh}

\newtheorem{thm}{Theorem}[section]
\newtheorem{lm}[thm]{Lemma}
\newtheorem{coro}[thm]{Corollary}
\newtheorem{prop}[thm]{Proposition}

\theoremstyle{definition}
\newtheorem{df}[thm]{Definition}
\newtheorem{ex}[thm]{Example}

\theoremstyle{remark}

\newtheorem*{rem}{Remark}
\newtheorem{ack}{Acknowledgments}

\def\BVi{\BV_\infty}
\def\IBLi{\IBL_\infty}
\def\Li{L_\infty}
\def\Ai{A_\infty}
\def\Z{\mathbb Z}
\def\g{\mathfrak g}

\newcommand{\abs}[1]{{\lvert#1\rvert}}

\begin{document}

\title[The BV formalism for $\Li$-algebras]{The BV formalism for $\Li$-algebras}

\author[D. Bashkirov]{Denis Bashkirov}
\email{bashk003@umn.edu}
\address {School of Mathematics\\University of Minnesota\\
  Minneapolis, MN 55455, USA}

\author[A. A. Voronov]{Alexander A. Voronov} \email{voronov@umn.edu}
\address {School of Mathematics\\University of Minnesota\\
  Minneapolis, MN 55455, USA, and Kavli IPMU (WPI), University of
  Tokyo, Kashiwa, Chiba 277-8583, Japan}

\thanks{This work was supported by the World Premier International
  Research Center Initiative (WPI Initiative), MEXT, Japan, the
  Institute for Mathematics and its Applications with funds provided
  by the National Science Foundation, and a grant from the Simons
  Foundation (\#282349 to A.~V.).}

\date{January 31, 2016}

\begin{abstract}
  Functorial properties of the correspondence between commutative
  $\BVi$-algebras and $\Li$-algebras are investigated. The category of
  $\Li$-algebras with $\Li$-morphisms is characterized as a certain
  category of \emph{pure} $\BVi$-algebras with \emph{pure}
  $\BVi$-morphisms. The functor assigning to a commutative
  $\BVi$-algebra the $\Li$-algebra given by higher derived brackets is
  also shown to have a left adjoint. Cieliebak-Fukaya-Latschev's
  machinery of $\IBLi$- and $\BVi$-morphisms is further developed with
  introducing the logarithm of a map.
\end{abstract}

\maketitle

\section*{Introduction}

One way to approach the notion of a strongly homotopy Lie algebra is
via the language of formal geometry, see
\cite{kontsevich-soibelman}. Namely, it is known that the data of an
$\Li$-algebra $\g$ is equivalent to that of a formal pointed
differential graded (dg) manifold $\g[1]$. The corresponding $\Li$
structure is then encoded in the cofree dg cocommutative coalgebra
$S(\g[1])$ of distributions on $\g[1]$ supported at the basepoint. The
idea of Batalin-Vilkovisky (BV) formalism in physics suggests that it
might be useful to study what the $\Li$ structure looks like from a
Fourier-dual perspective \cite{losev}, namely, the point of view of
the standard dg commutative algebra structure on $S(\g[1])$. In fact,
we show that an $\Li$ structure on $\g$ translates into a special kind
of commutative homotopy Batalin-Vilkovisky ($\BVi$) structure on
$S(\g[-1])$ and, moreover, does it in a functorial way. Geometrically,
we can say that we describe a formal pointed dg manifold $\g[1]$ as a
pointed $\BVi$-manifold $(\g[-1])^*$ of a special kind. This provides
a Fourier-dual, BV alternative to the standard characterization of the
category of $\Li$-algebras as a subcategory of the category of dg
cocommutative coalgebras or formal pointed dg manifolds. In
particular, the \emph{coalgebra codifferential} on $S(\g[1])$ encoding
the structure of an $\Li$-algebra on a graded vector space $\g$ turns
into a \emph{square-zero differential operator with linear
  coefficients} on the \emph{algebra} $S(\g[-1])$. We also show that
the functor that assigns to an $\Li$-algebra $\g$ the $\BVi$-algebra
$S(\g[-1])$ is left adjoint to a ``functor'' that assigns to a
$\BVi$-algebra the $\Li$-algebra given by higher derived
brackets. This fact may be interpreted geometrically as a statement
that the functor $\g[1] \mapsto (\g[-1])^*$ from formal pointed dg
manifolds to pointed $\BVi$-manifolds has a right adjoint.

The correspondence between $\Li$ and $\BVi$ structures that we
establish in the paper is to a large extent motivated by the technique
of higher derived brackets. The origins of higher derived brackets can
be traced back to the iterated commutators of A.~Grothendieck, see
Expos\'e VII$_{\text{A}}$ by P.~Gabriel in \cite{SGA3}, and
J.-L. Koszul \cite{koszul}, used in the algebraic study of
differential operators, though the subject really flourished later in
physics under the name of higher ``antibrackets'' in the works of
J.~Alfaro, I.~A. Batalin, K.~Bering, P.~H. Damgaard and R.~Marnelius
\cite{alfaro-damgaard,batalin-bering-damgaard,batalin-marnelius:q,
  batalin-marnelius:d,batalin-marnelius:g,bering-damgaard-alfaro,
  bering:ncbv} on the BV formalism. A mathematically friendly approach
was developed by F.~Akman's \cite{akman:bv,akman:master} and
generalized further by T.~Voronov \cite{tvoronov:hdb,tvoronov:hdba},
who described $\Li$ brackets derived by iterating a binary Lie bracket
not necessarily given by the commutator. There have been various
versions of higher derived brackets introduced in other contexts, such
as the ternary bracket of \cite{roytenberg-weinstein} for Courant
algebroids and its higher-bracket generalization of
\cite{fiorenza-manetti,getzler:hdb} for dg Lie algebras or the $\Ai$
products of \cite{boerjeson} for dg associative algebras and the
twisted $\Li$ brackets and $\Ai$ products of
\cite{markl:boerjeson}. The notion of a $\BVi$-algebra was studied by
K.~Bering and T.~Lada \cite{bering-lada:examples}, K.~Cieliebak and
J.~Latschev \cite{cieliebak-latschev}, I.~G\'alvez-Carrillo, A.~Tonks
and B.~Vallette \cite{galvez-carrillo-tonks-vallette}, O.~Kravchenko
\cite{kravchenko:bv}, and D.~Tamarkin and B.~Tsygan
\cite{tamarkin-tsygan}. Lately, considerable interest in homotopy
involutive Lie bialgebras, or $\IBLi$-algebras, which are
$\BVi$-algebras of a particular type, has emerged, see
\cite{cieliebak-latschev, cieliebak-fukaya-latschev,muenster-sachs,
  campos-merkulov-willwacher}. $\IBLi$-algebras are also known as
$\Omega(\operatorname{CoFrob})$-algebras of \cite{dc-t-t} and closely
related to quantum (or loop) homotopy Lie algebras of
\cite{zwiebach:intro,markl:loop}.

The BV formalism as a replacement of the dg-coalgebra language seems
to be even more natural for studying Lie-Rinehart pairs $(\g, A)$, see
\cite{huebschmann,vitagliano}.

We review the notion of a $\BVi$-algebra in Section 1 and describe the
$\BVi$ structure on $S(\g[-1])$ in Section 2. A characterization of
$\BVi$-algebras arising this way is presented in Section 3. In Section
4 we prove the first main result of the paper: a description of the
category of $\Li$-algebras as a certain subcategory of the category of
$\BVi$-algebras. Section 5 is dedicated to the second main result, the
adjunction theorem.

\subsection*{Conventions and notation}

We will work over a ground field $k$ of characteristic zero. A
differential graded (dg) vector space $V$ will mean a complex of
$k$-vector spaces with a differential of degree one. The degree of a
homogeneous element $v \in V$ will be denoted by $\abs{v}$. In the
context of graded algebra, we will be using the Koszul rule of signs
when talking about the graded version of notions involving symmetry,
including commutators, brackets, symmetric algebras, derivations,
\emph{etc}., often omitting the modifier \emph{graded}. For any
integer $n$, we define a \emph{translation} (or \emph{$n$-fold
  desuspension$)$} $V[n]$ of $V$: $V[n]^p := V^{n+p}$ for each $p \in
\Z$.  For a dg vector space $V$, will also consider the dg
$k[[\hbar]]$-module $V[[\hbar]]$ of formal power series in a variable
$\hbar$ of degree 2 with values in $V$. We will also sometimes refer
to differential operators of order $\le n$ as differential operators
of order $n$.

\begin{ack}
  The authors are grateful to Maxim Kontsevich, Yvette
  Kosmann-Schwarzbach, Janko Latschev, Martin Markl, Jim Stasheff,
  Daniel Sternheimer, Luca Vitagliano, Theodore Voronov, and the
  anonymous referee for useful remarks. A.~V. also thanks IHES, where
  part of this work was done, for its hospitality.
\end{ack}

\section{Homotopy BV algebras}

We will utilize a strictly commutative version of the notion of a
homotopy BV algebra, also known as a generalized BV algebra, due to
Kravchenko \cite{kravchenko:bv}, which is less general than the
full-fledged homotopy versions of \cite{tamarkin-tsygan} and
\cite{galvez-carrillo-tonks-vallette}. Nevertheless, we will take the
liberty to use the term $\BVi$-algebra, following a trend set by
several authors
\cite{cieliebak-latschev,terilla-tradler-wilson,braun-lazarev,vitagliano}. The
following definition gives a graded version of Grothendieck's notion
of a differential operator in commutative algebra.

\begin{df}
\label{do}
Let $n \ge 0$ be an integer. A $k$-linear operator $D: V \to V$ on a
graded commutative algebra $V$ is a said to be a \emph{differential
  operator of order $\le n$} if for any $n+1$ elements $a_0, \dots,
a_n \in V$, we have
\[
[[ \dots [D, L_{a_0}], \dots], L_{a_n} ] = 0,
\]
where the $L_a$ is the left-multiplication operator
\[
L_a (b) := ab
\]
on $V$ and the bracket $[-,-]$ is the graded commutator of two
$k$-linear operators.
\end{df}

\begin{df}
\label{BVi-df}
  A \emph{$\BVi$-algebra} is a graded commutative algebra $V$ over $k$
  with a $k$-linear map $\Delta: V \to V[[\hbar]]$ of degree one
  satisfying the following properties:
\[
\Delta = \frac{1}{\hbar} \sum_{n=1}^\infty \hbar^n \Delta_n,
\]
where $\Delta_n$ is a differential operator of order (at most) $n$ on
$V$,
\[
\Delta^2 = 0, \qquad \text{and} \qquad \Delta (1) = 0,
\]
The continuous (in the $\hbar$-adic topology), $k[[\hbar]]$-linear
extension of $\Delta$ to $V[[\hbar]]$ will also be denoted $\Delta:
V[[\hbar]] \to V[[\hbar]]$ and called a \emph{$\BVi$ operator}.
\end{df}

Recall that we assumed that $\abs{\hbar} = 2$, thus $\abs{\Delta_1} =
1$, $\abs{\Delta_2} = -1$, and generally, $\abs{\Delta_n} = 3 - 2n$
for $n \ge 1$. Note that $\Delta_1$ will automatically be a
differential in the usual sense, \emph{i.e.}, define the structure of
a dg commutative algebra on $V$. If $\Delta_n = 0$ for $n \ge 3$, we
recover the notion of a \emph{dg BV algebra}, see
\cite{ks:exact,akman:bv,bar-kon,huebschmann, manin:book,
  kravchenko:bv, tamarkin-tsygan}. If moreover $\Delta_1 = 0$, we get
the notion of a \emph{BV algebra}, also known as a
\emph{Beilinson-Drinfeld algebra}, see
\cite{beilinson-drinfeld,gwilliam,costello-gwilliam}. BV algebras
arose as part of the BV formalism in physics. A basic geometric
example of (a $\Z/2\Z$-graded version of) a BV algebra is the algebra
of functions on a smooth supermanifold with an odd symplectic form and
a volume density, see \cite{schwarz:bv,getzler:bv}. An example of such
a supermanifold is the odd cotangent bundle $\Pi T^*M$ of a
(classical, rather than super) smooth manifold $M$ with a volume form,
where $\Pi T^*M$ denotes the translation $T^*[-1]M$ modulo 2 of the
cotangent bundle. A Lie-theoretic version of this example is the
graded symmetric algebra $S(\g[-1])$, also known as the exterior
algebra $\Lambda (\g)$, of a Lie algebra $\g$, with the
Chevalley-Eilenberg differential as $\Delta_2$. We will describe an
$\Li$ generalization of this example is Section~\ref{l-to-bv}. More
generally, one can view the $\BVi$ structure considered in this paper
as a homotopy version of the algebraic structure arising in BV
geometry.

\begin{ex}(\cite{khudaverdian-voronov,braun-lazarev,vitagliano})
\label{dr-koszul}
Let $M$ be a smooth graded manifold and $C^\infty(M, \linebreak[0]
S(T[-1]M)[1])$ be the graded Lie algebra of (global, smooth)
multivector fields on $M$ with respect to the Schouten bracket. When
$M$ is a usual, ungraded manifold, $S(T[-1]M)[1]$ is the
exterior-algebra bundle $\bigwedge TM$, in which a $k$-vector field,
or a section of $\bigwedge^k TM$, has degree $k-1$. A
\textit{generalized Poisson structure} on a graded manifold $M$ is a
multivector field $P$ of degree one such that $[P,P]=0$. A generalized
Poisson structure on $M$ may be expanded as $P=P_0 + P_1+\dots$ with
$P_n\in C^\infty(M, S^{n}(T[-1]M)[1])$. For $n \ge 1$, the generalized
Lie derivative $\Delta_n=[d,\iota_{P_n}]$, where $\iota_{(-)}$ is the
interior product, defines an $n$th-order differential operator of
degree $3 - 2n$ on the de Rham algebra $(\Omega(M),d)$, where
$\Omega(M) := C^\infty(M, S(T^*[-1]M))$. If we assume that $P_{0} = 0$
to avoid differential operators of order zero, then $\Delta = (d +
\Delta_1) + \Delta_2 \hbar+\dots+\Delta_n\hbar^{n-1} + \dots:
\Omega(M)\to \Omega(M)[[\hbar]]$ defines a $\BVi$ structure on
$\Omega(M)$, known as the \emph{de Rham-Koszul} $\BVi$ structure.
\end{ex}

\section{From homotopy Lie algebras to homotopy BV algebras}
\label{l-to-bv}

The construction of this section belongs essentially to C.~Braun and
A.~Lazarev, see \cite[Example 3.12]{braun-lazarev}. Consider an
\emph{$\Li$-algebra} $\g$, \emph{i.e.}, a graded vector space $\g$ and
a codifferential on the cofree graded cocommutative coalgebra
$S(\g[1])$ on $\g[1]$ with respect to the ``shuffle''
comultiplication:
\[
\delta (x_1 \dots x_m) := \sum_{n=0}^m \sum_{\sigma \in \Sh_{n,m-n}}
(-1)^{\abs{x_\sigma}} (x_{\sigma(1)} \dots x_{\sigma(n)}) \otimes
(x_{\sigma(n+1)} \dots x_{\sigma(m)}),
\]
where $x_1, \dots, x_m \in \g[1]$, $\Sh_{n,m-n}$ is the set of
$(n,m-n)$ \emph{shuffles}: permutations $\sigma$ of $\{1,2, \dots,
m\}$ such that $\sigma(1) < \sigma(2) < \dots < \sigma (n)$ and
$\sigma(n+1) < \sigma(n+2) < \dots < \sigma (m)$, and
$(-1)^{\abs{x_\sigma}}$ is the \emph{Koszul sign} of the permutation
of $x_1 \dots x_m$ to $x_{\sigma(1)} \dots x_{\sigma(m)}$ in
$S(\g[1])$.  Here a \emph{codifferential} is a coderivation $D:
S(\g[1]) \to S(\g[1])$ of degree one such that $D^2=0$ and $D(1) =
0$. Given that a coderivation is determined by its projection to the
cogenerators, we can write
\[
D = D_1 + D_2 + D_3 + \dots,
\]
where $D_n$ is the extension as a coderivation of the $n$th symmetric
component $l_n: \nolinebreak[4] S^n(\g[1]) \to \g[1]$ of the
projection $S(\g[1]) \xrightarrow{D} S(\g[1]) \to \g[1]$. An explicit
relation between $D_n$ and $l_n$ will be useful: for $x_1, \dots, x_m
\in \g[1]$
\begin{equation}
\label{coderivation}
D_n (x_1 \dots x_m) = \sum_{\sigma \in \Sh_{n,m-n}} (-1)^{\abs{x_\sigma}}
l_n(x_{\sigma(1)}, \dots, x_{\sigma(n)}) x_{\sigma(n+1)} \dots
x_{\sigma(m)},
\end{equation}
if $m \ge n$, and $D_n (x_1 \dots x_m) = 0$ otherwise.

\begin{thm}[C.~Braun and A.~Lazarev]
  Given an $\Li$-algebra $\g$, the free graded commutative algebra
  $S(\g[-1])$ becomes a $\BVi$-algebra under the $\BVi$ operator
\begin{equation}
\label{BVi-structure}
\Delta := \frac{1}{\hbar} \sum_{n=1}^\infty \hbar^n D_n.
\end{equation}
\end{thm}

\begin{proof}
\begin{sloppypar}
  Since $D_n: S^m(\g[1]) \to S^{m-n+1}(\g[1])$ is a degree one map, it
  turns into a map $D_n: S^m(\g[-1]) \to S^{m-n+1}(\g[-1])$ of degree
  $3-2n$ under the new grading.\footnote{Strictly speaking, the use of
    $D_n$ to denote the two maps is abuse of notation, because they
    differ by powers of the double suspension operator $\g[1] \to
    \g[-1]$, but we prefer to keep it this way, because double
    suspension does not affect signs.}  For each $n$,
\[
\sum_{i+j = n} D_i D_j = 0,
\]
because this sum is exactly the component of $D^2$ which maps
$S^m(\g[1])$ to $S^{m-n+2} (\g[1])$. The map $D_n$ will also be a
differential operator of order $n$, because of the following lemma,
which expresses an important relation between coderivations and
differential operators and may be observed directly from
Equation~\eqref{coderivation}.
\end{sloppypar}

\begin{lm}
\label{prop}
The coderivation of the coalgebra $S(\g[1])$ extending a linear map
$S^n(\g[1]) \to \g[1]$ becomes a differential operator of order $\le
n$ on the algebra $S(\g[-1])$.
\end{lm}
\end{proof}

The construction of this section seems to be math-physics folklore in
the case when $(\g, d, [-,-])$ is a dg Lie algebra: the differential
$\Delta = D_1 + \hbar D_2$ defines a dg BV algebra structure on
$S(\g[-1])$. The operator $\Delta$ is essentially the homological
Chevalley-Eilenberg differential:
\begin{eqnarray*}
  \Delta (x_1 \dots x_m) & =& \sum_{i=1}^m (-1)^{\abs{x_1
      \dots
      x_{i-1}}} x_1 \dots dx_i \dots x_m \\
  &  + & \hbar \sum_{1 \le i<j \le m} (-1)^{\abs{x_{\sigma(i,j)}}+\abs{x_i}}
  [x_i, x_j] x_1 \dots
  \widehat{x_i} \dots \widehat{x_j} \dots x_m,
\end{eqnarray*}
where $\sigma(i,j)$ is the corresponding shuffle, the $x_i$'s in $\g$
are treated as elements of $\g[-1]$, and, following standard
conventions, $d = l_1$ and $[x_i, x_j] = (-1)^{\abs{x_i}} l_2(x_i,
x_j)$.

\begin{rem}
  An $A_\infty$-analogue of the above construction has been proposed
  by J.~Terilla, T.~Tradler, and S.~Wilson in
  \cite{terilla-tradler-wilson}: for an $\Ai$-algebra $V$, the tensor
  algebra $T(V[-1])$ (equipped with the shuffle product) is provided
  with a $\BVi$-structure. There is also an interesting generalization
  to the $\infty$-version of a Lie-Rinehart pair considered by
  L.~Vitagliano \cite{vitagliano}.
\end{rem}

\begin{rem}
  We will also need a certain $\hbar$-enhancement of the construction
  of a $\BVi$-algebra from an $\Li$-algebra. Suppose the graded
  $k[[\hbar]]$-module $\g[[\hbar]]$ for a graded vector space $\g$ is
  provided with the structure of a topological $\Li$-algebra over
  $k[[\hbar]]$ with respect to $\hbar$-adic topology. Then the same
  formula \eqref{BVi-structure} defines a $\BVi$-structure on
  $S(\g[-1])$ over $k$. There is a subtlety, though: each operator
  $D_n$ is a formal power series in $\hbar$ now, and in the
  $\hbar$-expansion
\[
\Delta = \frac{1}{\hbar} \sum_{n=1}^\infty \hbar^n \Delta_n,
\]
there are contributions to $\Delta_n$ from $D_1$, $D_{2}$, \dots, and
$D_n$. This still guarantees that $\Delta_n$ is a differential
operator of order at most $n$ on $S(\g[-1])$ satisfying the conditions
of Definition~\ref{BVi-df}.
\end{rem}

To summarize, given an $\Li$-algebra $\g$, we obtain a canonical
$\BVi$-algebra structure on $S(\g[-1])$. There is also a construction
going in the opposite direction.

\section{From homotopy BV algebras to homotopy Lie algebras}
\label{bv-to-l}

Suppose we have a $\BVi$-algebra $V$. Then for each $n \ge 1$, the
following \emph{higher derived brackets}
\begin{eqnarray}
\label{higher-brackets-h}
l^\hbar_n (a_1, \dots, a_n) & := & [[
\dots [\Delta, L_{a_1}], \dots], L_{a_n} ] 1\\
\nonumber
& = & \sum_{k=n}^\infty \hbar^{k-1}
[[ \dots [\Delta_k, L_{a_1}], \dots], L_{a_n} ] 1
\end{eqnarray}
on $V[[\hbar]]$, their $\hbar$-modification
\begin{equation}
\label{modified}
L_n := \frac{1}{\hbar^{n-1}} l^\hbar_n,
\end{equation}
and their ``semiclassical limit''
\begin{eqnarray}
\label{higher-brackets}
l_n (a_1, \dots, a_n) & := & \lim_{\hbar \to 0} L_n (a_1, \dots, a_n)\\
\nonumber
& = & \lim_{\hbar \to 0} \frac{1}{\hbar^{n-1}} [[
\dots [\Delta, L_{a_1}], \dots], L_{a_n} ] 1\\
\nonumber
& = & [[ \dots [\Delta_n, L_{a_1}], \dots], L_{a_n} ] 1
\end{eqnarray}
on $V$ turn out to be $\Li$ brackets, according to the results in this
section below.  Just for comparison, note that $l^\hbar_1 = L_1 =
\Delta$, whereas $l_1 = \Delta_1$. Observe also that we have a linear
(or strict) $\Li$-morphism
\begin{eqnarray*}
(V[[\hbar]][-1], l^\hbar_n) & \to & (V[[\hbar]][1], L_n),\\
v & \mapsto & \hbar v,
\end{eqnarray*}
which becomes an $\Li$-isomorphism after localization in
$\hbar$. Thus, we can think of the $\Li$ structure given by the
brackets $L_n$ as an $\hbar$-translation of the $\Li$ structure given
by $l^\hbar_n$.

One can express $\Delta$ through $l^\hbar_n$ via the following useful
formula
\begin{equation}
\label{magic}
\Delta(a_1 \dots a_n) = \sum_{j=1}^n \sum_{\sigma \in \Sh_{j,n-j}}
(-1)^{\abs{a_\sigma}} l^\hbar_j (a_{\sigma(1)}, \dots, a_{\sigma(j)})
a_{\sigma(j+1)} \dots a_{\sigma(n)}
\end{equation}
for $a_1, \dots, a_n \in V$, which is easy to prove by induction on
$n$ using Equation~\eqref{deviation} below, starting with $n=1$ for
$l_1^\hbar = \Delta$.

\begin{thm}[Bering-Damgaard-Alfaro]
\label{bda}
For a $\BVi$-algebra $V$, the higher brackets $l^\hbar_n$, $ n \ge 1$,
defined by \eqref{higher-brackets-h} endow the suspension
$V[[\hbar]][-1]$ with the structure of an $\Li$-algebra over
$k[[\hbar]]$. Moreover, the bracket $l^\hbar_{n+1}$ measures the
deviation of $l^\hbar_n$ from being a multiderivation with respect to
multiplication.
\end{thm}

\begin{rem}
  This result was first observed by the physicists
  \cite{bering-damgaard-alfaro} and proven in a more general context
  by T.~Voronov \cite{tvoronov:hdb,tvoronov:hdba}. The $\Li$ structure
  was also rediscovered by O.~Kravchenko in \cite{kravchenko:bv}.
\end{rem}

\begin{proof}
  Using the Jacobi identity for the commutator of linear operators
  along with the fact that $L_a$ and $L_b$ (graded) commute, it is
  easy to check that the higher brackets $l^\hbar_n$ are symmetric on
  $V[[\hbar]]$:
\[
l^\hbar_n(a_{\pi(1)}, \dots, a_{\pi(n)}) = (-1)^{\abs{a_\pi}}
l^\hbar_n (a_1, \dots, a_n)
\]
for all $a_1, \dots, a_n \in V[[\hbar]]$, where $(-1)^{\abs{a_\pi}}$
is the Koszul sign, see Section~\ref{l-to-bv}. Since $\abs{\Delta} =
1$, the degree of $l^\hbar_n$ as a bracket on $V[[\hbar]]$ will be the
same. We can extend the $k[[\hbar]]$-linear operators $l^\hbar_n:
S^n(V)[[\hbar]] \to V[[\hbar]]$ to coderivations $D_n: S(V)[[\hbar]]
\to S(V)[[\hbar]]$ and consider the total coderivation
\[
D = D_1 + D_2 + \dots.
\]
on $S(V)[[\hbar]]$. The differential property $D^2=0$ for this
coderivation is equivalent to the series of \emph{higher Jacobi
  identities}:
\[
\sum_{n=1}^m \sum_{\sigma \in \Sh_{n,m-n}} (-1)^{\abs{a_\sigma}}
l^\hbar_{m-n+1}(l^\hbar_n(a_{\sigma(1)}, \dots, a_{\sigma(n)}),
a_{\sigma(n+1)}, \dots, a_{\sigma(m)}) = 0
\]
for all $a_1, \dots, a_m \in V[[\hbar]]$, $m \ge 1$. The physicists
\cite{bering-damgaard-alfaro} and T.~Voronov \cite{tvoronov:hdb} in a
more general situation checked these identities using the following
key observation for an arbitrary odd operator $\Delta$ on
$V[[\hbar]]$, not necessarily squaring to zero:
\begin{multline*}
  \sum_{n=1}^m \sum_{\sigma \in \Sh_{n,m-n}} (-1)^{\abs{a_\sigma}}
  l^\hbar_{m-n+1}(l^\hbar_n(a_{\sigma(1)}, \dots, a_{\sigma(n)}),
  a_{\sigma(n+1)},
  \dots, a_{\sigma(m)}) \\
  = [[ \dots [\Delta^2, L_{a_1}], \dots], L_{a_m} ] 1.
\end{multline*}
Given that $\Delta^2=0$, the higher Jacobi identities follow.

The deviated multiderivation property, more precisely,
\begin{multline}
\label{deviation}
l^\hbar_{n+1} (a_1, \dots, a_i, a_{i+1}, \dots, a_{n+1})
= l^\hbar_{n} (a_1, \dots, a_i \cdot a_{i+1}, \dots, a_{n+1})
\\- (-1)^{(1+\abs{a_1} + \dots \abs{a_{i-1}})\abs{a_i}}
a_{i} l^\hbar_{n} (a_1, \dots, a_{i-1}, a_{i+1}, \dots, a_{n+1})
\\- (-1)^{(1+\abs{a_1} + \dots \abs{a_{i}})\abs{a_{i+1}}}
a_{i+1} l^\hbar_{n} (a_1, \dots, a_{i}, a_{i+2}, \dots, a_{n+1})
\end{multline}
of the higher brackets may be derived from the identity
\[
[Q, L_{ab} ] = [[Q,L_a],L_b] + (-1)^{\abs{Q}\cdot \abs{a}} L_a
[Q,L_b] + (-1)^{(\abs{Q}+\abs{a})\abs{b}} L_b [Q,L_a]
\]
for an arbitrary (homogeneous) linear operator $Q$ on
$V[[\hbar]]$. Applying this to $Q = [[ \dots [\Delta, L_{a_1}], \dots
], L_{a_n}]$, we see that $l^\hbar_{n+1}$ measures the deviation of
$l^\hbar_n$ from being a derivation in the last variable. Since the
higher brackets are symmetric, we obtain the same property in each
variable.
\end{proof}

\begin{coro}[T.~Voronov \cite{tvoronov:hdb}]
\label{tv}
  Given a $\BVi$-algebra $V$, the brackets $L_n$, $n \ge 1$, defined
  by \eqref{modified} endow the graded $k[[\hbar]]$-module
  $V[[\hbar]][1]$ with the structure of an $\Li$-algebra over
  $k[[\hbar]]$. Likewise, the brackets $l_n$, $n \ge 1$, defined by
  \eqref{higher-brackets} endow the graded vector space $V[1]$ with
  the structure of an $\Li$-algebra over $k$. Moreover, the brackets
  $l_n$ are multiderivations of the graded commutative algebra
  structure.
\end{coro}

\begin{proof}
  The statements of the corollary are obtained as the ``semiclassical
  limit'' of the statements of Theorem~\ref{bda}, and so is the
  proof. Note the change of suspension to desuspension from the
  theorem to the corollary. This corresponds intuitively to the
  statement that the semiclassical limit of the space $V[[\hbar]][-1]$
  is $\hbar V[-1] = V[1]$. Concretely, the desuspension guarantees
  that the degree of the $n$th higher bracket $l_n$ on $V[2]$ is still
  one: indeed $\abs{\Delta_n} = 3 -2n$, when $\Delta_n$ is considered
  as an operator on $V$; therefore, the degree of $l_n$ as a
  multilinear operation on $V[2]$ will be $3-2n + 2(n-1) = 1$.

  The multiderivation property is obtained by dividing
  \eqref{deviation} by $\hbar^{n-1}$ and noticing that the left-hand
  side will not survive the limit as $\hbar \to 0$, because it has
  $\hbar$ as a factor.
\end{proof}

\begin{rem}
  The algebraic structure which combines the graded commutative
  multiplication with the $\Li$ structure given by the brackets $l_n$
  is a particular case of the $G_\infty$-algebra structure, see
  \cite{getzler-jones,tamarkin1998,tamarkin:thesis,voronov:hga}.
\end{rem}

\begin{rem}
  The construction given by the higher brackets $l_n$ obviously
  induces an operad morphism $s\Li \to \BVi$, where $s$ denotes the
  operadic suspension, see, \emph{e.g}.,
  \cite{markl-shnider-stasheff}. Here $\BVi$ stands for the operad
  describing commutative $\BVi$-algebras, as opposed to the full
  $\BVi$ operad of
  \cite{tamarkin-tsygan,galvez-carrillo-tonks-vallette,dc-v}. This
  operad morphism immediately gives a functor from the category of
  $\BVi$-algebras to that of $\Li$-algebras, provided we restrict
  ourselves to morphisms of algebras over operads, \emph{i.e}.,
  consider only linear (strict) morphisms. However, we will focus on
  nonlinear morphisms in the subsequent sections of the paper.
\end{rem}

\begin{ex}
  The $\Li$ structure given by the brackets $l_n$ coming from the
  $\BVi$ structure of Example~\ref{dr-koszul} is known as the de
  Rham-Koszul $\Li$ structure and generalizes the Koszul brackets on
  the de Rham complex of a manifold,
  \cite{khudaverdian-voronov,braun-lazarev,vitagliano}.
\end{ex}

\begin{ex}
  Let $\g$ be an $\Li$-algebra. Then by the construction of
  Section~\ref{l-to-bv}, we get the structure of a $\BVi$-algebra on
  $S(\g[-1])$. If we apply the ``semiclassical'' derived brackets
  $l_n$ of this section to the $\BVi$-algebra $S(\g[-1])$, we will get
  the structure of an $\Li$-algebra on $S(\g[-1])[1]$. Later we show
  in Theorem~\ref{BVi-L}\eqref{3} that the $\Li$-algebra $\g$ becomes
  an $\Li$-subalgebra of $S(\g[-1])[1]$. The higher brackets on
  $S(\g[-1])[1]$ may be viewed as extensions of the higher brackets on
  $\g$ as multiderivations. These higher brackets \emph{generalize the
    Schouten bracket} on the exterior algebra of a Lie algebra to the
  $\Li$ case.
\end{ex}

Our goal is to characterize those $\BVi$-algebras which come from
$\Li$-algebras as in Section~\ref{l-to-bv}. Note that such a
$\BVi$-algebra is free as a graded commutative algebra by
construction: $V = S(U)$, and that for each $n \ge 1$, the $n$th
component $\Delta_n$ of the $\BVi$ operator maps $S^m(U)$ to
$S^{m-n+1}(U)$ for $m \ge n$ and to 0 for $0 \le m < n$, because of
Equation~\eqref{coderivation}. Since an $n$th-order differential
operator on a free algebra $S(U)$ is determined by its restriction to
$S^{\le n}(U)$, this condition on $\Delta_n$ is equivalent to the
condition that $\Delta_n$ maps $S^n(U)$ to $U$ and $S^{<n}(U)$ to
0. Interpreting differential operators on $S(U)$ as linear
combinations of partial derivatives with polynomial coefficients,
differential operators of the above type may also be characterized as
\emph{differential operators of order $n$ with linear coefficients}.

\begin{df}
A \emph{pure} $\BVi$-algebra is the free graded commutative algebra
$S(U)$ on a graded vector space $U$ with a $\BVi$ operator $\Delta:
S(U) \to S(U)[[\hbar]]$ such that, for each $n \ge 1$, $\Delta_n$ maps
$S^n(U)$ to $U$ and $S^{<n}(U)$ to 0.
\end{df}

The following theorem (Parts \eqref{1} and \eqref{2}) shows that
freeness and purity are not only necessary but also sufficient
conditions for a $\BVi$-algebra to arise from an $\Li$-algebra.

\begin{thm}
\label{BVi-L}
\begin{enumerate}
\item
\label{1}
Given a pure $\BVi$ algebra $(V = S(U), \Delta)$, the restriction of
the brackets $l_n$ to $U[1] \subset S(U)[1]$ provides $U[1]$ with the
structure of an $\Li$-subalgebra.
\item
\label{2}
The original pure $\BVi$ structure on $S(U)$ coincides with the $\BVi$
structure \eqref{BVi-structure} of Section~$\ref{l-to-bv}$ coming from
the derived $\Li$ structure on $U[1]$.
\item 
\label{3}
If we start with an $\Li$ structure on a graded vector space $U[1]$
and construct the $\BVi$-algebra $S(U)$ as in Section~$\ref{l-to-bv}$,
then the derived brackets $l_n$ on $U[1] \subset S(U)[1]$ return the
original $\Li$ structure on $U[1]$.
\end{enumerate}
\end{thm}

\begin{proof}
  The first statement we need to check is that $l_n(a_1, \dots, a_n)$
  is in $U$ whenever $a_1, \dots, a_n \in U$ and $n \ge 1$, as a
  priori all we know is that $l_n(a_1, \dots, a_n) \in S(U)$. The
  condition that $\Delta_n$ maps $S^m(U)$ to 0 for $0 \le m < n$
  implies by \eqref{higher-brackets} that $l_n(a_1, \dots, a_n) =
  \Delta_n(a_1 \dots a_n)$, which must be in $U$, because of the
  condition $\Delta_n: S^n(U) \to S^1(U) = U$.

  For the second statement, we need to check that the $n$th-order
  differential operator $\Delta_n$, the $n$th component of the given
  $\BVi$ structure, is equal to the coderivation $D_n$ defined by
  \eqref{coderivation}. Recall that on the free algebra $S(U)$, an
  $n$th-order differential operator is determined by its restriction
  to $S^{\le n}(U)$. Given the assumption that $\Delta_n$ vanishes on
  $S^{< n}(U)$, it follows that $\Delta_n$ on $S(U)$ is determined by
  its restriction to $S^n(U)$. By the previous paragraph, its
  restriction to $S^n(U)$ is equal to $l_n$. On the other hand, this
  is also the restriction of the coderivation $D_n$ to $S^n(U)$, as
  per formula \eqref{coderivation}.  Lemma~\ref{prop} shows that the
  coderivation $D_n$ is also an $n$th-order differential
  operator. Thus, it is also determined by its restriction to
  $S^n(U)$.

  Finally, let $l_n$ be the $\Li$ brackets on an $\Li$-algebra $U[1]$
  and $\tilde{l}_n$ be the higher derived brackets produced on the
  pure $\BVi$-algebra $S(U)$ by formula \eqref{higher-brackets} for
  $n=1$, $2$, \dots We claim that $\tilde{l}_n (x_1,\dots,x_n) =l_n
  (x_1,\dots,x_n)$ for all $n$ and $x_1,\dots,x_n\in U$. Indeed,
  \[
  \tilde{l}_n(x_1,\dots,x_n)=[[ \dots [D_n, L_{x_1}], \dots], L_{x_n}
  ] 1,
  \]
  where $D_n$ is the extension of $l_n$ to $S(U)$ as a coderivation,
  see Equation~\eqref{coderivation}. The same equation implies that
  $D_n:S(U)\to S(U)$ is zero on $S^{<n}(U)$. Hence all but one term
  $(D_n\circ L_{x_1}\circ\dots\circ L_{x_n})(1)$ of this iterated
  commutator vanish.  It remains to observe that by
  \eqref{coderivation} this is nothing but $l_n(x_1,\dots,x_n)$.
\end{proof}

\begin{rem}
  A general, not necessarily pure $\BVi$-structure on $S(U)$ leads to
  an interesting algebraic structure on $U[1]$, called an
  \emph{involutive $\Li$-bialgebra} or an $\IBLi$-\emph{algebra}, see
  \cite{cieliebak-fukaya-latschev,muenster-sachs}. From the
  properadic, rather than BV prospective, this structure is described
  in \cite{galvez-carrillo-tonks-vallette} and \cite{dc-t-t}. The BV
  formalism for ordinary $\Li$-bialgebras seems to be subtler:
  apparently, one needs to weaken the definition of a $\BVi$ structure
  on $S(U)$ by requiring that for each $n \ge 1$, the coefficient by
  $\hbar^{n-1}$ in the expansion of $\Delta^2 = \frac{1}{2} [\Delta,
  \Delta]$ be of order $\le n-1$, rather than the expected order $\le
  n$, instead of asking for vanishing of each coefficient of
  $\Delta^2$ in its expansion in $\hbar$.
\end{rem}

\section{Functoriality}
\label{functor}

The correspondence between $\BVi$-algebras and $\Li$-algebras
that we studied above has remarkable functorial properties with a
suitable notion of a morphism between $\BVi$-algebras.

First of all, recall the definition of a morphism between
$\Li$-algebras.

\begin{df}
  An $\Li$-\emph{morphism} $\g \to \g'$ between $\Li$-algebras is a
  morphism $S(\g [1]) \to S(\g' [1])$ of codifferential graded
  coalgebras, \emph{i.e.}, a morphism of graded coalgebras commuting
  with the structure codifferentials, such that $1 \in S^0(\g[1])$
  maps to $1 \in S^0(\g'[1])$.
\end{df}

\begin{rem}
  Since we deal with counital coalgebras, we assume that
  $\Li$-morphisms respect the counits. The extra condition $1 \mapsto
  1$ means that we are talking about ``pointed'' morphisms, if we
  invoke the interpretation of $\Li$-morphisms as morphisms between
  formal pointed dg manifolds, see \cite{kontsevich-soibelman}.
\end{rem}

Now we will consider the corresponding notion of a morphism between
$\BVi$-algebras. We will only need this notion for $\BVi$-algebras of
Theorem~\ref{BVi-L}, that is to say, $\BVi$-algebras which are
pure. Somewhat more generally, we will give a definition in the case
when the source $\BVi$-algebra is just free. A more general notion of
a $\BVi$-morphism for more general $\BVi$-algebras can be found in
\cite{tamarkin-tsygan}. We use the definition of a $\BVi$-morphism by
Cieliebak-Latchev \cite{cieliebak-latschev}.

Before giving the definition, we need to recall a few more
notions. Fix a morphism $f: A \to A'$ between graded commutative algebras. We
say that a $k$-linear map $D: A \to A'$ is a \emph{differential
  operator of order $\le n$ over $f: A \to A'$} or simply a
\emph{relative differential operator of order $\le n$} if for any
$n+1$ elements $a_0, \dots, a_n \in A$, we have
\[
[[ \dots [D, L_{a_0}], \dots], L_{a_n} ] = 0,
\]
where $[D,L_a]$ is understood as the map $A \to A'$ defined by
\[
[D,L_a] (b) := D(ab) - (-1)^{\abs{a} \cdot \abs{D}} f(a) D(b).
\]
For $f = \id$ we recover the standard definition Def.~\ref{do} of a
differential operator on a graded commutative algebra.

Let $V=S(U)$ be a free graded commutative algebra and $V'$ an
arbitrary graded commutative algebra. Given a $k$-linear map $\varphi:
S(U) \to V'[[\hbar]]$ of degree zero such that $\varphi(1) = 0$,
define a degree-zero, continuous $k[[\hbar]]$-linear map
$\exp(\varphi): S(U)[[\hbar]] \to V'[[\hbar]]$, called the
\emph{exponential}, by
\begin{multline*}
\exp(\varphi) (x_1 \dots x_m) := \\ \sum_{k=1}^m \frac{1}{k!}  \! \!
\! \! \! \sum_{\substack{i_1,\dots,i_k = 1\\i_1 + \dots + i_k = m}}^{m} \!
\! \! \!  \frac{1}{i_1! \dots i_k!}  \! \! \! \sum_{\sigma \in S_m}
(-1)^{\abs{x_\sigma}} \varphi(x_{\sigma(1)} \dots x_{\sigma(i_1)})
\dots \varphi(x_{\sigma(m-i_k+1)} \dots x_{\sigma(m)}),
\end{multline*}
where $S_m$ denotes the symmetric group, $x_1, \dots, x_m$ are in $U$,
and $(-1)^{\abs{x_\sigma}}$ is the Koszul sign of the permutation of
$x_1 \dots x_m$ to $x_{\sigma(1)} \dots x_{\sigma(m)}$ in $S(U)$. By
convention, we set $\exp(\varphi) (1) := 1$. The reason for the
exponential notation, introduced by Cieliebak and Latschev
\cite{cieliebak-latschev}, is, perhaps, the following statement, which
they might have been aware of,
cf.\ \cite{cieliebak-fukaya-latschev}. The proof is a straightforward
computation.

\begin{lm}
\label{exp}
If $S \in \lambda U[[\hbar]]^0[[\lambda]]$ or $\lambda
U((\hbar))^0[[\lambda]]$, where $\lambda$ is another, degree-zero
formal variable, then
\[
\exp(\varphi) (e^S) = e^{\varphi(e^S)}.
\]
Here we have extended $\varphi$ and $\exp(\varphi)$ to $\lambda
S(U)((\hbar))[[\lambda]]$ by $\hbar^{-1}$- and $\lambda$-linearity and
continuity.
\end{lm}

\begin{rem}
  The extra formal variable $\lambda$ in the lemma guarantees
  ``convergence'' of the exponential $e^S$. We could have achieved the
  same goal, if we considered completions of our algebras or assumed
  that $\lambda$ was a nilpotent variable, varying over the maximal
  ideals of finite-dimensional local Artin algebras. Informally
  speaking, given the way the space $S(U)[\lambda, \hbar, \hbar^{-1}]$
  of $S(U)$-valued polynomials in $\lambda$ and Laurent polynomials in
  $\hbar$ is completed: $S(U)((\hbar))[[\lambda]]$, we could think of
  $\lambda$ as being ``much smaller'' than $\hbar$.
\end{rem}

The exponential, not surprisingly, has an inverse, called the
\emph{logarithm}, which we will use a little later. Given a $k$-linear
map $\Phi: S(U) \to V'[[\hbar]]$ of degree zero such that $\Phi(1) =
1$, define a degree-zero, continuous $k[[\hbar]]$-linear map
$\log(\Phi): S(U)[[\hbar]] \to V'[[\hbar]]$ by
\begin{multline*}
\log(\Phi) (x_1 \dots x_m) \\ := \sum_{k=1}^m \frac{(-1)^{k-1}}{k} \!
\!  \! \sum_{\substack{i_1,\dots,i_k = 1\\i_1 + \dots + i_k = m}}^{m}
\!  \! \! \!  \frac{1}{i_1! \dots i_k!}  \! \sum_{\sigma \in S_m}
(-1)^{\abs{x_\sigma}} \Phi(x_{\sigma(1)} \dots x_{\sigma(i_1)})
\dots\\ \Phi(x_{\sigma(m-i_k+1)} \dots x_{\sigma(m)})
\end{multline*}
under the same notation as for the exponential. By convention, we set
$\log(\Phi) (1) := 0$.

\begin{df}[Cieliebak-Latchev \cite{cieliebak-latschev}]
\label{CL}
  A $\BVi$-\emph{morphism} from a $\BVi$-al\-ge\-bra $(V=S(U), \Delta)$
  to a $\BVi$-algebra $(V', \Delta')$ is a $k$-linear map $\varphi: V
  \to V'[[\hbar]]$ of degree zero satisfying the following properties:
\begin{enumerate}
\item $\varphi(1) = 0$,
\item $\exp(\varphi) \Delta = \Delta' \exp(\varphi)$, and
\item $\varphi$ admits an expansion
\[
\varphi = \frac{1}{\hbar} \sum_{n=1}^\infty \hbar^n \varphi_n,
\]
where $\varphi_n: V \to V'$ is a differential operator of order $\le
n$ over the morphism $S(U) \to V'$ induced by the zero linear map $U
\xrightarrow{0} V'$, \emph{i.e}., $\varphi_n$ maps $S^{>n}(U)$ to $0$.
\end{enumerate}
\end{df}
We will use the same notation for the continuous, $k[[\hbar]]$-linear
extension $\varphi: V[[\hbar]] \to V'[[\hbar]]$ of the $k$-linear map
$\varphi: V \to V'[[\hbar]]$, as well as for the corresponding
$\BVi$-morphism $\varphi: (V, \Delta) \to (V', \Delta')$.

\begin{rem}
  A $\BVi$-morphism can be regarded as a quantization of a morphism of
  dg commutative algebras. Indeed, $\varphi_1$ can be nonzero only on
  $U = S^1(U) \subset S(U)$ and by construction $\exp(\varphi_1)$ will
  be a graded algebra morphism. The equation $\exp(\varphi) \Delta =
  \Delta' \exp(\varphi)$ at $\hbar=0$ reduces to $\exp (\varphi_1)
  \Delta_1 = \Delta'_1 \exp(\varphi_1)$, which implies that
  $\exp(\varphi_1)$ is a morphism of dg algebras with respect to the
  ``classical limits'' $\Delta_1$ and $\Delta'_1$ of the $\BVi$
  operators.
\end{rem}

\begin{ex}
\label{multiplication}
A nice example of a $\BVi$-morphism $S(V) \to V$ may be obtained from
the projection $p_1: S(V) \to V$ of the symmetric algebra $S(V)$ to
its linear component $V = S^1(V)$ for any $\BVi$-algebra $V$. Before
talking about morphisms, we need to provide $S(V)$ with the structure
of a $\BVi$-algebra. To do that, we take the $\Li$ structure on
$V[[\hbar]][1]$ over $k[[\hbar]]$ given by the brackets $L_n$, see
\eqref{modified}, and then the $\BVi$ structure on $S(V)$ from the
remark at the end of Section~\ref{l-to-bv}. To regard $p_1$ as a
$\BVi$-morphism, we compose it with the inclusion $V \subset
V[[\hbar]]$ and get a $k$-linear map $\varphi = \varphi_1: S(V) \to
V[[\hbar]]$. By construction, $\varphi(1) = 0$. One can easily check
that $\exp(\varphi) = m$, the multiplication operator $S(V) \to V$. To
see that $\exp(\varphi)$ commutes with the $\BVi$ operators, we
observe that, for $a_1, \dots, a_n \in V$, the value of the $\BVi$
operator coming from the brackets $L_j$ on the product $a_1 \otimes
\dots \otimes a_n \in S(V)$ is equal to
\[
\sum_{j=1}^n \hbar^{j-1} \sum_{\sigma \in \Sh_{j,n-j}}
(-1)^{\abs{a_\sigma}} L_j(a_{\sigma(1)}, \dots, a_{\sigma(j)}) \otimes
a_{\sigma(j+1)} \otimes \dots \otimes a_{\sigma(n)},
\]
because of Equations \eqref{coderivation} and
\eqref{BVi-structure}. When we apply $m$ to that, the tensor product
(multipliciation in $S(V)$) will change to multiplication in $V$. The
result will just be equal to $(\Delta m) (a_1 \otimes \dots \otimes
a_n)$ in view of Equation~\eqref{magic}.
\end{ex}

Another feature of a $\BVi$-morphism $\varphi: S(U)[[\hbar]] \to
V'[[\hbar]]$ is that it propagates solutions $S \in \lambda
U((\hbar))^2[[\lambda]]$ of the \emph{Quantum Master Equation
  $($QME$\, )$}
\begin{equation}
\label{QME}
\Delta e^{S/\hbar} = 0
\end{equation}
to solutions of the QME in $\lambda V'((\hbar))^2 [[\lambda]]$.

\begin{prop}
If $\varphi: S(U) \to V'$ is a $\BVi$-morphism and $S \in \lambda
U((\hbar))^2[[\lambda]]$ is a solution of the QME $\eqref{QME}$, then
\[
S' := \hbar \varphi(e^{S/\hbar}) \in \lambda V'((\hbar))^2 [[\lambda]]
\]
is a solution of the QME
\[
\Delta' e^{S'/\hbar} = 0.
\]
\end{prop}

\begin{proof}
  By Lemma~\ref{exp} we have $e^{\varphi(e^{S/\hbar})} =
  \exp(\varphi)(e^{S/\hbar})$. Since $\exp(\varphi)$ must respect the
  $\BVi$ operators, we get
\[
\Delta' e^{\varphi(e^{S/\hbar})} = \Delta' \exp(\varphi)(e^{S/\hbar})
= \exp(\varphi) \Delta (e^{S/\hbar}) = 0.
\]
\end{proof}

Now we are ready to study functorial properties of the correspondence
between $\Li$-algebras and $\BVi$-algebras from
Theorem~\ref{BVi-L}. Since the $\BVi$-algebra corresponding to an
$L_\infty$-algebra is pure, we would like to concentrate on
$\BVi$-morphisms between such $\BVi$-algebras. Among these
$\BVi$-morphisms, those of the following type turn out to form a
category.

\begin{df}
  We will call a $\BVi$-morphism $\varphi: S(U) \to S(U')$ between
  $\BVi$-algebras which are free as graded commutative algebras
  \emph{pure}, if $\varphi_n$ maps $S^n (U)$ to $U' \subset
  S(U')[[\hbar]]$ and all other symmetric powers $S^k(U)$ to 0. In
  other words, one can say that $\varphi_n$ is a \emph{differential
    operator of order $n$ with linear coefficients over the morphism}
  $S(U) \to S(U')$ induced by the zero map $U \xrightarrow{0} S(U')$.
\end{df}

\emph{$\BVi$-algebras which are free as graded commutative algebras}
form a \emph{category} under pure $\BVi$-morphisms in the following
way. Given pure $\BVi$-morphisms $V \xrightarrow{\varphi} V'
\xrightarrow{\psi} V''$, their composition $\psi \diamond \varphi: V
\to V''$ is defined by composing their exponentials:
\[
\psi \diamond \varphi := \log (\exp (\psi) \circ \exp (\varphi)).
\]
Under this composition, the role of identity morphism on $S(U)$ is
played by $\varphi = \varphi_1 = \id_U$: in this case, $\exp (\varphi)
= \id_{S(U)}$.

\begin{prop}
The composition $\psi \diamond \varphi$ of any pure $\BVi$-morphisms
is a pure $\BV_\infty$-morphism.
\end{prop}

\begin{proof}
  First of all, we need to see that the properties (1)-(3) of a
  $\BVi$-morphism are satisfied. Property (1) is satisfied because of
  our conventions on the values of exponentials and logarithms of maps
  at 1. Property (2) is obvious by construction. Property (3) may be
  established from the following formula for pure morphisms:
\begin{multline}
\label{comp}
(\psi \diamond \varphi) (x_1 \dots x_m) \\ = \sum_{k=1}^m \frac{1}{k!}
\!  \!  \! \sum_{\substack{i_1,\dots,i_k = 1\\i_1 + \dots + i_k =
    m}}^{m} \!  \! \! \!  \frac{1}{i_1! \dots i_k!}  \! \sum_{\sigma
  \in S_m} (-1)^{\abs{x_\sigma}} \psi(\varphi(x_{\sigma(1)} \dots
x_{\sigma(i_1)}) \dots\\ \varphi(x_{\sigma(m-i_k+1)} \dots
x_{\sigma(m)})),
\end{multline}
which is easily verified by exponentiating it and comparing it to
$\exp (\psi) \circ \exp (\varphi)$. Indeed, the coefficient $(\psi
\diamond \varphi)_n (x_1 \dots x_m)$ by $\hbar^{n-1}$ on the
right-hand side will be coming from terms
\[
\psi_j(\varphi_{j_1} (x_{\sigma(1)} \dots x_{\sigma(i_1)}) \dots
\varphi_{j_k} (x_{\sigma(m-i_k+1)} \dots x_{\sigma(m)}))
\]
with $j-1+ \sum_{p=1}^k (j_p -1) = j-1+ \sum_{p=1}^k j_p - k =
n-1$. Observe that because of Property (3) for $\psi$ and $\varphi$,
for such a term not to vanish, it is necessary that $j \ge k$ and $j_p
\ge i_p$ for each $p$. Thus, we will have $n = j+ \sum_{p=1}^k j_p - k
\ge k + \sum_{p=1}^k i_p - k = m$, which is Property (3) for $\psi
\diamond \varphi$.  The fact that the composite $\BVi$-morphism is
pure is obvious from Eq.~\eqref{comp} and purity of $\psi$.
\end{proof}

\begin{thm}
\label{equiv}
The correspondence $\g \mapsto S(\g[-1])$ of Section~$\ref{l-to-bv}$
from $\Li$-algebras to $\BVi$-algebras is functorial. This functor
establishes an equivalence between the category of $\Li$-algebras and
the full subcategory of pure $\BVi$-algebras of the category of
$\BVi$-algebras free as graded commutative algebras with pure
morphisms. The functor $V = S(U) \mapsto U[1]$ of
Theorem~$\ref{BVi-L}$\eqref{1} provides a weak inverse to this
equivalence.
\end{thm}

Restricting this theorem to the world of dg Lie algebras and dg BV
algebras, we obtain the following corrolaries, which, surprisingly,
seem to be new.

\begin{coro}
  The functor $\g \mapsto S(\g[-1])$ from dg Lie algebras to dg BV
  algebras establishes an equivalence between the category of dg Lie
  algebras with $\Li$-morphisms and the category of dg BV algebras
  $(V, \Delta_1, \Delta_2)$, free as graded commutative algebras
  $V=S(U)$ and whose BV structure is pure: $\Delta_2$ maps $U$ to 0
  and $S^2(U)$ to $U$, with $\BVi$-morphisms $S(U) \to S(U')$
  satisfying the purity condition.
\end{coro}

\begin{coro}
  The functor $\g \mapsto S(\g[-1])$ from dg Lie algebras to dg BV
  algebras establishes an equivalence between the category of dg Lie
  algebras (with dg Lie morphisms) and the category of dg BV algebras
  $(V, \Delta_1, \Delta_2)$, free as graded commutative algebras $V =
  S(U)$ and whose BV structure is pure: $\Delta_2$ maps $U$ to 0 and
  $S^2(U)$ to $U$, with morphisms defined as morphisms $\Phi: S(U) \to
  S(U')$ of graded algebras respecting the differentials $\Delta_1$
  and $\Delta_2$ and satisfying the purity condition: $\Phi$ maps $U$
  to $U'$.
\end{coro}

Now let us prove the theorem.

\begin{proof}
  We need to see that an $\Li$-morphism $\g \to \g'$ induces a
  $\BVi$-morphism $S(\g[-1]) \to S(\g'[-1])$. By definition an
  $\Li$-morphism is a graded coalgebra morphism $\Phi: S(\g[1]) \to
  S(\g'[1])$ respecting the codifferentials and such that $\Phi(1) =
  1$. As a coalgebra morphism, $\Phi$ is determined by its projection
  $\varphi: S(\g[1]) \to \g'[1]$ to the cogenerators $\g'[1]$ via the
  following formula:
\begin{multline}
\label{Phi}
\Phi (x_1 \dots x_m) = \\
\sum_{k=1}^m \frac{1}{k!}  \! \!
  \! \sum_{\substack{i_1,\dots,i_k = 1\\i_1 + \dots + i_k = m}}^{m}
\sum_{\sigma \in \Sh_{i_1, \dots, i_k}} (-1)^{\abs{x_\sigma}}
  \varphi(x_{\sigma(1)} \dots x_{\sigma(i_1)}) \dots
  \varphi(x_{\sigma(m-i_k+1)} \dots x_{\sigma(m)}),
\end{multline}
where $\Sh_{i_1, \dots, i_k}$ denotes the set of $(i_1, \dots, i_k)$
shuffles, $x_1, \dots, x_m$ are in $\g[1]$, and
$(-1)^{\abs{x_\sigma}}$ is the Koszul sign of the permutation of $x_1
\dots x_m$ to $x_{\sigma(1)} \dots x_{\sigma(m)}$ in $S(\g[1])$. (For
$m=0$, we just have $\Phi(1) = 1$ and $\varphi(1) = 0$.)  The above
formula follows from iteration of the coalgebra morphism property:
\[
\delta^{k-1} \Phi = \Phi^{\otimes k} \delta^{k-1}
\]
along with its projection to $(\g'[1])^{\otimes k}$ for each $k = 1,
\dots, m$. To turn $\varphi$ into a $\BVi$-morphism, we need to
rewrite it as a power series in $\hbar$:
\begin{equation}
\label{phi-h}
\varphi_\hbar := \frac{1}{\hbar} \sum_{n=1}^\infty \hbar^n \varphi_n,
\end{equation}
where $\varphi_n: S(\g[-1]) \to S(\g'[-1])$ maps all symmetric powers
to $0$, except for $S^n(\g[-1])$, on which $\varphi_n$ is the
restriction of $\varphi: S(\g[1]) \to \g'[1]$ to $S^n(\g[1])$ along
with an appropriate shift in degree to make it into a linear map
$S^n(\g[-1]) \to \g'[-1]$. Note that the degree of $\varphi$ was
supposed to be zero, as it was a projection of the morphism $\Phi$ of
graded coalgebras. In terms of grading on $S^n(\g[-1])$ and $\g'[-1]$,
the degree of shifted $\varphi_n$ is $2-2n$. Multiplication by
$\hbar^{n-1}$ shifts that degree back to $0$, thus we see that the
degree of $\varphi_\hbar$ is zero as well.

Note that by construction, the purity condition on $\varphi_\hbar$ is
satisfied, and thereby we have
\begin{multline*}
\exp(\varphi_\hbar) (x_1 \dots x_m) \\ = \sum_{k=1}^m
\frac{\hbar^{m-k}}{k!}  \! \! \! \! \! \!
\sum_{\substack{i_1,\dots,i_k = 1\\i_1 + \dots + i_k = m}}^{m} \!  \!
\! \!  \frac{1}{i_1! \dots i_k!}  \! \sum_{\sigma \in S_m}
(-1)^{\abs{x_\sigma}} \varphi (x_{\sigma(1)} \dots x_{\sigma(i_1)})
\dots \\ \varphi (x_{\sigma(m-i_k+1)} \dots x_{\sigma(m)}),
\end{multline*}
whence, comparing this to the right-hand side of \eqref{Phi}, we get
\[
\exp(\varphi_\hbar) = \sum_{m=0}^\infty \hbar^{m} \Phi_m,
\]
where $\Phi_{m}$ is the component of $\Phi$ of degree $-m$ in the
grading given by the symmetric power, so as
\[
\Phi = \sum_{m=0}^\infty \Phi_{m}.
\]

\begin{sloppypar}
We know that $\Phi$ is compatible with the structure codifferentials
$D$ and $D'$ of $\g$ and $\g'$: $\Phi D = D' \Phi$. The $\BVi$
operator on $S(\g[-1])$ was defined as $\Delta = \sum_{m =1}^\infty
\hbar^{m-1} D_m$, where $D_m$ maps each $S^n(\g[1])$ to
$S^{n-m+1}(\g[1])$; likewise for $S(\g'[-1])$, see
\eqref{BVi-structure}. Thus, the equation $\exp(\varphi_\hbar) \Delta
= \Delta' \exp(\varphi_\hbar)$ is satisfied, being just a weighted sum
of the components of the equation $\Phi D = D' \Phi$, where the
component shifting the symmetric power down by $n \ge 0$ is being
multiplied by $\hbar^n$. This completes verification of the fact that
$\varphi_\hbar$ is a pure $\BVi$-morphism.
\end{sloppypar}

Conversely, we need to see that every pure $\BVi$-morphism comes from
an $\Li$-morphism. By Theorem~\ref{BVi-L} we can assume that the
source and the target of this $\BVi$-morphism are the $\BVi$-algebras
$S(\g[-1])$ and $S(\g'[-1])$ coming from some $\Li$-algebras $\g$ and
$\g'$. Every $\BVi$-morphism is given by a formal $\hbar$-series like
\eqref{phi-h} satisfying the three conditions of
Definition~\ref{CL}. Since the morphism is pure, we can ``drop'' the
$\hbar$ from $\varphi_\hbar$ and note that the formal series
\[
\varphi := \sum_{n=1}^\infty \varphi_n
\]
will produce a well-defined linear map $S(\g[1]) \to \g'[1]$. Dropping
the $\hbar$ results in this map also having degree zero. Now we can
generate a unique morphism $\Phi: S(\g[1]) \to S(\g'[1])$ of
coalgebras by the linear map $\varphi$. This morphism $\Phi$ will be
given by formula \eqref{Phi}. Since $\varphi_h (1) = 0$, we get
$\varphi(1) = 0$ and $\Phi(1) = 1$ by the same formula. We just need
to check that this morphism $\Phi$ respects the codifferentials $D$
and $D'$ on these two coalgebras, respectively. As in the first part
of the proof, we see that the equation $\exp(\varphi_\hbar) \Delta =
\Delta' \exp(\varphi_\hbar)$ implies $\Phi D = D' \Phi$. Thus, $\Phi$
is an $L_\infty$-morphism.

We also need to check the functoriality properties of the
correspondence $\g \mapsto S(\g[-1])$. The fact that $\id_{\g}$ maps
to the identity morphism is obvious. Now, if we have two
$\Li$-morphisms $\g \to \g' \to \g''$ given by dg coalgebra morphisms
$S(\g[1]) \xrightarrow{\Phi} S(\g'[1]) \xrightarrow{\Psi} S(\g''[1])$
with $\Phi = \sum_{m=0}^\infty \Phi_m$ and $\Psi = \sum_{m=0}^\infty
\Psi_m$, we note that the exponentials $\exp(\varphi_\hbar) =
\sum_{m=0}^\infty \hbar^m \Phi_m$ and $\exp(\psi_\hbar) =
\sum_{m=0}^\infty \hbar^m \Psi_m$ of the respective $\BVi$-morphisms
will compose in the same way as $\Phi$ and $\Psi$, the only difference
being that the component decreasing the symmetric power by $m$ gets
multiplied by $\hbar^m$.
\end{proof}

\section{Adjunction}

In this section, we prove a certain ``adjunction'' property. The
quotation marks are due to the fact that in our setting arbitrary
$\BVi$-algebras do not even make up a category. However, in the
theorem below, only $\BVi$-morphisms from a domain of the form
$S(\g[-1])$, with $\g$ being an $\Li$-algebra, play a role.

Recall that given an $\Li$-algebra $\g$, we have constructed a
$\BVi$-algebra $S(\g[-1])$ in Section~\ref{l-to-bv}. Conversely, given
a $\BVi$-algebra $V$, we have used the higher derived brackets $L_n$
to induce an $\Li$-structure on $V[[\hbar]][1]$ over $k[[\hbar]]$ as
in Corollary~\ref{tv}.

Note that both constructions are functorial. The fact that $\g \mapsto
S(\g[-1])$ defines a functor is the first statement of
Theorem~\ref{equiv}. We need to see that the construction assigning to
a $\BVi$-algebra $V$ the $\Li$-algebra $(V[[\hbar]][1], L_n)$ is also
functorial. Given a $\BVi$-morphism $\varphi: V = S(U) \to V'$, we
need to construct an $\Li$-morphism $V[[\hbar]][1]\to
V'[[\hbar]][1]$. This construction will be accomplished in two steps.
\smallskip

\noindent
\emph{Step} 1. Compose the $\BVi$-morphism $\varphi: V \to V'$ with
the $\BVi$-morphism $p_1: S(V) \to V$ of Example~\ref{multiplication}
to get a $\BVi$-morphism $\varphi \diamond p_1: S(V) \to V'$.
\smallskip

\noindent
\emph{Step} 2. Given an $\Li$-algebra $\g[[\hbar]]$ over $k[[\hbar]]$
and a $\BVi$-morphism $\psi: S(\g[-1]) \to V'$, where $S(\g[-1])$ is
provided with the $\BVi$ structure of the remark at the end of
Section~\ref{l-to-bv}, we will construct a canonical $\Li$-morphism
$\g[[\hbar]] \to V'[[\hbar]][1]$. Then we will just apply this
construction to the $\BVi$-morphism $S(V) \to V'$ of Step 1.

In order to construct an $\Li$-morphism $\g[[\hbar]] \to
V'[[\hbar]][1]$, take the graded $k[[\hbar]]$-coalgebra morphism,
continuous in the $\hbar$-adic topology,
\[
F: S(\g[1])[[\hbar]] \to S(V'[2])[[\hbar]]
\]
induced by the $k[[\hbar]]$-linear map
\[
f: S(\g[1])[[\hbar]] \to V'[[\hbar]][2]
\]
whose restriction $f|_{S^k(\g[1])[[\hbar]]}: S^k(\g[1])[[\hbar]] \to
V'[[\hbar]][2]$ is the restriction of $\hbar^{1-k} \psi$ to
$S^k(\g[1])[[\hbar]$ for $k \ge 0$:
\[
f|_{S^k(\g[1])[[\hbar]]} = \hbar^{1-k} \psi|_{S^k(\g[1])[[\hbar]]}.
\]
This map takes values in $V'[[\hbar]][2]$, despite the division by a
power of $\hbar$, because the restriction of $\psi$ to $S^k (\g[-1])$
is in fact equal to $\sum_{n=k}^\infty \hbar^{n-1} \psi_n =
\hbar^{k-1} \sum_{n=0}^\infty \hbar^{n} \psi_{n+k}$. Note that since
$\psi$ is of degree zero, $f$ will also have degree zero.

We need to check that $F$ defines an $\Li$-morphism. It is easy to see
that $F(1) = 1$, because $\psi (1) = 0$. What is far less trivial is
the fact that $F$ respects the codifferentials, the structure
codifferential $D$ on $S(\g[1])[[\hbar]]$ and the codifferential $D'$
on $S(V'[2])[[\hbar]]$ induced as a continuous coderivation, see
\eqref{coderivation}, by the sum of the brackets \eqref{modified}:
\[
L_n: S^n(V'[2])[[\hbar]] \to V'[[\hbar]][2].
\]
What we know is $\Delta' \exp (\psi) = \exp (\psi) \Delta$, where
$\Delta'$ is the $\BVi$ operator on $V'$ and $\Delta$ is the structure
codifferential $D$ on $S(\g[1])[[\hbar]]$ enhanced by $\hbar$, as in
the remark at the end of Section~\ref{l-to-bv}. To see that this
implies the equation $D' F = F D$, we need to develop some BV calculus
and compare it to colagebra calculus.

Let us start with coalgebra calculus. Each side of the equation is a
continuous coderivation over the coalgebra morphism $F$ and as such
determined by the projection $p_1: S(V'[2])[[\hbar]] \to
V'[[\hbar]][2]$ to the cogenerators $V'[[\hbar]][2]$ of the
range. Thus, all we need to show is that $p_1 D' F = p_1 F D$, after
projecting to the cogenerators. Now, for a monomial $x_1 \dots x_m \in
S^m(\g[1])$, we have
\begin{multline}
\label{D-F}
  p_1 D' F (x_1 \dots x_m) \\
  = \sum_{k=1}^m \frac{1}{k!}  \! \!  \! \sum_{\substack{i_1,\dots,i_k
      = 1\\i_1 + \dots + i_k = m}}^{m} \sum_{\sigma \in \Sh_{i_1,
      \dots, i_k}} (-1)^{\abs{x_\sigma}} L_k(f(x_{\sigma(1)} \dots
  x_{\sigma(i_1)}), \dots,\\
  f(x_{\sigma(m-i_k+1)} \dots x_{\sigma(m)})),
\end{multline}
using the shuffle notation, see Equation~\eqref{Phi}, as well as
\begin{equation}
\label{F-D}
p_1 F D (x_1 \dots x_m) = f (D(x_1 \dots x_m)).
\end{equation}
We need to show that the right-hand sides of these equations are
equal, based on the equation $\Delta' \exp (\psi) = \exp (\psi)
\Delta$. We will do that after we develop some BV calculus.

Turning to BV calculus, we have
\begin{multline}
\label{coderivation-delta}
\Delta (x_1 \dots x_m) \\
= \sum_{k=1}^m \hbar^{k-1} \sum_{\tau \in \Sh_{k,m-k}}
(-1)^{\abs{x_\tau}} l_k(x_{\tau(1)}, \dots, x_{\tau(k)}) x_{\tau(k+1)}
\dots x_{\tau(m)},
\end{multline}
where $l_k$'s are the $\Li$ brackets on $\g$, because of
Equation~\eqref{coderivation}. Now apply $\exp(\psi)$ to both sides,
reassemble products of $\psi$'s not containing $l_k$'s into $\exp
(\psi)$, and use \eqref{coderivation-delta} again to pass from $l_k$'s
back to $\Delta$ and get
\begin{multline}
\label{exp-delta}
\exp (\psi) \Delta (x_1 \dots x_m) \\
= \sum_{n=1}^m \sum_{\sigma \in \Sh_{n,m-n}} (-1)^{\abs{x_\sigma}}
\psi(\Delta(x_{\sigma(1)} \dots x_{\sigma(n)})) \exp (\psi)
(x_{\sigma(n+1)} \dots x_{\sigma(m)}).
\end{multline}

Move on to computation of $\Delta' \exp(\psi)$:
\begin{multline}
\label{delta-exp}
  \Delta' \exp(\psi) (x_1 \dots x_m) \\
  = \sum_{n=1}^m \sum_{\sigma \in \Sh_{n,m-n}} (-1)^{\abs{x_\sigma}}\\
  \sum_{k=1}^n \frac{1}{k!}  \! \!  \! \sum_{\substack{i_1,\dots,i_k =
      1\\i_1 + \dots + i_k = n}}^{n} \sum_{\tau \in \Sh_{i_1, \dots,
      i_k}} (-1)^{\abs{x_{\tau \sigma}}} l^\hbar_k(\psi(x_{\tau
    \sigma(1)} \dots
  x_{\tau \sigma(i_1)}), \dots,\\
  \psi(x_{\tau \sigma(n-i_k+1)} \dots x_{\tau \sigma(n)})) \exp
  (\psi) (x_{\sigma(n+1)} \dots x_{\sigma(m)}),
\end{multline}
which follows from the definition of $\exp(\psi)$ and the identity
\eqref{magic}.

Now let us compare \eqref{exp-delta} with \eqref{delta-exp}, which are
equal by assumption. One can show by induction on $m$ that the top,
$n=m$ terms of the two formulas must also be equal:
\begin{multline*}
  \psi(\Delta(x_{1} \dots x_{m})) \\ =
\sum_{k=1}^m \frac{1}{k!}  \! \!  \! \sum_{\substack{i_1,\dots,i_k =
      1\\i_1 + \dots + i_k = m}}^{m} \sum_{\tau \in \Sh_{i_1, \dots,
      i_k}} (-1)^{\abs{x_{\tau}}} l^\hbar_k(\psi(x_{\tau(1)} \dots
  x_{\tau(i_1)}), \dots,\\
  \psi(x_{\tau(m-i_k+1)} \dots x_{\tau(m)})).
\end{multline*}
It remains to pass from $\psi$, $\Delta$, and $l^\hbar_k$ to $f$, $D$,
and $L_k$, respectively, in this equation, with appropriate powers of
$\hbar$, resulting in the equation
\begin{multline*}
  \hbar^{m-1} f(D(x_{1} \dots x_{m})) \\= \hbar^{m-1} \sum_{k=1}^m
  \frac{1}{k!}  \! \!  \! \sum_{\substack{i_1,\dots,i_k = 1\\i_1 +
      \dots + i_k = m}}^{m} \sum_{\tau \in \Sh_{i_1, \dots, i_k}}
  (-1)^{\abs{x_{\tau}}} L_k(f(x_{\tau(1)} \dots
  x_{\tau(i_1)}), \dots,\\
  f(x_{\tau(m-i_k+1)} \dots x_{\tau(m)})).
\end{multline*}
In view of \eqref{D-F} and \eqref{F-D}, we see that $D' F = F D$.
This completes Step 2.

\begin{thm}
  Suppose $\g$ is an $\Li$-algebra and $V$ is a $\BVi$-algebra. There
  exists a canonical bijection
\[
\Hom_{\BVi} (S(\g[-1]), V) \cong \Hom_{\Li} (\g, V[[\hbar]][1]),
\]
where the $\Li$-structure on $V[[\hbar]][1]$ is given by the modified
brackets $L_n$. This bijection is natural in the $\Li$-algebra $\g$
and in the $\BVi$-algebra $V$.
\end{thm}

\begin{proof}
  A correspondence from the $\BVi$-morphisms on the left-hand side to
  the $\Li$-morphisms on the right-hand side was constructed in Step 2
  before the theorem in a more general case of an $\Li$-algebra over
  $k[[\hbar]]$.

  Conversely, given an $\Li$-morphism $F: S(\g[1]) \to
  S(V[2])[[\hbar]]$, we use the same conversion formula
\begin{equation}
  \label{conversion}
\varphi|_{S^k(\g[1])[[\hbar]]} = \hbar^{k-1} f|_{S^k(\g[1])[[\hbar]]},
\end{equation}
$f$ being the projection of $F$ to the cogenerators $V[2][[\hbar]]$,
for $k \ge 0$, as in Step~2 before the theorem, to get a
$\BVi$-morphism $\varphi: S(\g[-1]) \to V$. Tracing the argument there
backward, we see that $\varphi$ is indeed a $\BVi$-morphism. This
establishes a bijection in the adjunction formula.

The naturality of the construction follows from the fact that, in view
of \eqref{conversion}, $F$ and $\exp(\varphi)$ are given by almost
identical formulas, with the only difference coming from insertion of
powers of $\hbar$, which plays the role of grading shift.
\end{proof}

\begin{coro}
  The functor $\g \mapsto S(\g[-1])$ of Sections $\ref{l-to-bv}$ and
  $\ref{functor}$ from the category of $\Li$-algebras to the category
  of $\BVi$-algebras free as graded commutative algebras with pure
  morphisms has a right adjoint, which is given by the functor of
  modified higher derived brackets $L_n$.
\end{coro}

\begin{rem}
  This corollary generalizes the construction of a pair of adjoint
  functors by Beilinson and Drinfeld \cite[4.1.8]{beilinson-drinfeld}
  from the case of dg Lie and BV algebras to the case of $\Li$- and
  $\BVi$-algebras.
\end{rem}

\begin{rem}
  The results of this section extend easily to the case when an
  $\Li$-algebra $\g$ is replaced with a topological $\Li$-algebra
  $\g[[\hbar]]$ over $k[[\hbar]]$ and we use the $\BVi$ structure on
  $S(\g[-1])$ described in the remark at the end of
  Section~\ref{l-to-bv}. In particular, there is a natural bijection
\[
\Hom_{\BVi} (S(\g[-1]), V) \cong \Hom_{\Li} (\g[[\hbar]], V[[\hbar]][1]),
\]
where on the right-hand side, we consider continuous $\Li$-morphisms
over $k[[\hbar]]$.
\end{rem}

\bibliographystyle{amsalpha}
\bibliography{bvi,gmp}

\end{document}